\documentclass[11pt, final]{article}
\usepackage{a4}
\usepackage{amsmath}
\usepackage{amstext}
\usepackage{amssymb}
\usepackage{showkeys}
\usepackage{epsfig}
\usepackage{cite}
\newtheorem{theorem}{Theorem}
\newtheorem{df}{Definition}
\newenvironment{definition}{\df\rm}{\enddf}
\newtheorem{ex}{Example}
\newenvironment{example}{\ex\rm}{\endex}
\newtheorem{lemma}[theorem]{Lemma}
\newtheorem{proposition}[theorem]{Proposition}
\newtheorem{rem}{Remark}
\newenvironment{remark}{\rem\rm}{\endrem}

\newcounter{unnumber}

\newtheorem{prf}[unnumber]{Proof.}
\newenvironment{proof}{\prf\rm}{\hfill{$\blacksquare$}\endprf}
\newtheorem{as}[unnumber]{Assumption}
\newenvironment{assumption}{\as\rm}{\endas}
\newcommand{\R}{\mathbb{R}}%

\DeclareMathOperator*\inte{int}%
\DeclareMathOperator*\ri{ri}%
\DeclareMathOperator*\cl{cl}%
\DeclareMathOperator*\epi{epi}%
\DeclareMathOperator*\dom{dom}%

\DeclareMathOperator*\qi{qi}%
\DeclareMathOperator*\qri{qri}%
\DeclareMathOperator*\sqri{sqri}%
\DeclareMathOperator*\cone{cone}%
\DeclareMathOperator*\core{core}%
\DeclareMathOperator*\co{co}%
\DeclareMathOperator*\B{\overline{\R}}%
\DeclareMathOperator*\cvar{CVaR}
\DeclareMathOperator*\var{VaR}
\DeclareMathOperator*\essup{essup}%
\DeclareMathOperator*\esinf{essinf}%

\title{Looking for appropriate qualification conditions for subdifferential formulae and dual representations for convex risk measures}

\author{Radu Ioan Bo\c{t} \thanks
{Faculty of Mathematics, Chemnitz University of Technology,
D-09107 Chemnitz, Germany, e-mail:
radu.bot@mathematik.tu-chemnitz.de. Research partially supported by DFG (German Research Foundation), project WA 922/1-3.} \and Alina-Ramona Fr\u atean
\thanks {Faculty of Mathematics and Computer Science, Babe\c{s}-Bolyai University, Cluj-Napoca,
Romania, e-mail: alina-ramona.fratean@s2009.tu-chemnitz.de. Research done during the stay of the author in the academic year 2009/2010 at Chemnitz University of Technology
as a guest of the Chair of Applied Mathematics (Approximation Theory). The author wishes to thank for the financial support provided from programs co-financed by The Sectoral Operational Programme Human Resources Development,
Contract POSDRU 6/1.5/S/3 –- ``Doctoral studies: through science towards society''.}}

\begin{document}
\maketitle

\noindent \textbf{Abstract.} A fruitful idea, when providing subdifferential formulae and dual representations for convex risk measures, is to  make use of the
conjugate duality theory in convex optimization. In this paper we underline the outstanding role played by the qualification conditions in the context of different problem formulations in this area. We show that not only the meanwhile classical generalized interiority point ones come here to bear, but also a recently introduced one formulated by means of the quasi-relative interior.\\

\noindent \textbf{Key Words.} convex risk measures, optimized certainty equivalent, monotone and cash-invariant hulls, qualification conditions\\

\noindent \textbf{AMS subject classification.} 49N15, 90C25, 90C46, 91B30

\section{Introduction and preliminaries}

Let ${\cal X}$ be a separated locally convex vector space and ${\cal X}^*$ its
topological dual space. We denote by $\langle x^*,x\rangle$ the
value of the linear continuous functional $x^*\in {\cal X}^*$ at $x\in {\cal X}$.

For a subset $C$ of ${\cal X}$ we denote by $\co C$, $\cl C$ and
$\inte C$ its \emph{convex hull}, \emph{closure}
and \emph{interior}, respectively. The set $\cone
C:=\cup_{\lambda \geq 0}\lambda C$ denotes the \emph{cone
generated by $C$}, while the \emph{normal cone} of $C$ at $x\in C$ is
given by $N_C(x)=\{x^*\in {\cal X}^*:\langle x^*,y-x\rangle\leq 0 \ \forall y\in C\}$. When $C$ is a convex and closed set, by $C_\infty:=\{x \in {\cal X}: x+C \subseteq C\}$, which is in this case a convex closed cone, we denote the \emph{asymptotic cone} of $C$.

The \emph{indicator function} of a set $C\subseteq {\cal X}$, denoted by
$\delta_C$, is defined by
$\delta_C:{\cal X}\rightarrow\B:=\R\cup\{\pm\infty\}$,
$$\delta_C(x)=\left\{
\begin{array}{ll}
0, & \mbox {if } x\in C,\\
+\infty, & \mbox{otherwise}.
\end{array}\right.$$
For a function $f:{\cal X}\rightarrow\B$ we denote by $\dom f=\{x\in
{\cal X}:f(x)<+\infty\}$ its \emph{effective domain} and by $\epi f=\{(x,r)\in
{\cal X}\times\R:f(x)\leq r\}$ its \emph{epigraph}. We call $f$
\emph{proper} if $\dom(f)\neq\emptyset$ and $f(x)>-\infty$ for all
$x\in {\cal X}$. The \emph{Fenchel-Moreau conjugate} of $f$ is the function
$f^*:{\cal X}^*\rightarrow\B$ defined by
$$f^*(x^*)=\sup\limits_{x\in {\cal X}}\{\langle
x^*,x\rangle-f(x)\}\,\forall x^*\in {\cal X}^*.$$ Similarly, when ${\cal X}^*$ is endowed with the weak$^*$ topology, the
\emph{biconjugate function} of $f,$ $f^{**}:{\cal X}\rightarrow\B$, is given
by
$$f^{**}(x)=\sup\limits_{x^*\in {\cal X}^*}\{\langle
x^*,x\rangle-f^*(x^*)\}\,\forall x\in {\cal X}.$$ By the \emph{Fenchel-Moreau
Theorem}, whenever $f:{\cal X}\rightarrow\B$ is a proper, convex and lower
semicontinuous function, one has $f=f^{**}.$

For $f:{\cal X}\rightarrow\B$ an arbitrary function the set
$\partial f(x)=\{x^*\in {\cal X}^*: f(y)-f(x)\geq\langle x^*, y-x \rangle \ \forall y\in
{\cal X}\}$, when $f(x)\in\mathbb{R}$, denotes the \emph{subdifferential} of $f$
at $x$, while if $f(x)\in \{\pm\infty\}$ we take by convention
$\partial f(x)=\emptyset$. Regarding a function and its conjugate we have the \emph{Young-Fenchel inequality}
$f^*(x^*)+f(x)\geq\langle x^*,x\rangle$ for all $x\in {\cal X}$ and $x^*\in
{\cal X}^*.$ Moreover, for all $x\in {\cal X}, x^*\in {\cal X}^*$ one has
\begin{equation}\label{subdiff-conjug-zero}f^*(x^*)+f(x) = \langle x^*,x\rangle \Leftrightarrow x^*\in
\partial f(x).\end{equation}

If $f:{\cal X}\rightarrow\B$ is a proper, convex and lower semicontinuous function, then
by $f_\infty :{\cal X}\rightarrow\B$ we denote the \emph{recession function} of $f$, which is defined as being the function whose epigraph
is $(\epi f)_\infty$. The recession function is in this setting a proper, sublinear and lower semicontinuous function, while for all $d \in {\cal X}$
one has
$$f_{\infty}(d)=\sup\{f(x+d)-f(x):x \in \dom f\}$$
and (see, for instance, \cite{Zal-carte})
\begin{equation}\label{asymptotic_func}
f_{\infty}(d)=\lim_{\substack{t\to+\infty}}\frac{f(x+td)-f(x)}{t}=\sup_{\substack{t>0}}\frac{f(x+td)-f(x)}{t} \ \forall x\in\dom f.
\end{equation}

Having $f_i:{\cal X}\rightarrow\B, i=1,...,m,$ given proper functions we denote by $f_1\square...\square f_m :{\cal X}\rightarrow\B,$
$f_1\square...\square f_m(x)=\inf\{\sum_{i=1}^m f_i(x_i) :  \sum_{i=1}^m x_i = x\}$, for $x\in {\cal X}$, their \emph{infimal convolution}.

In the formulation of the qualification conditions which we employ in the investigations made in this paper we will make use of several generalized interiority notions.
For a convex set $C\subseteq {\cal X}$, we recall those interiority notions we need in the following:
\begin{enumerate}
\item[$\bullet$] the \emph{algebraic interior} or \emph{core} of $C$ (cf. \cite{Zal-carte}),\\
$\core C=\{x\in C: \cone(C-x)={\cal X} \};$
\item[$\bullet$] the \emph{strong quasi-relative interior} of $C$ (cf. \cite{Borwein-Jeyak-Lewis-Wolk, Zal-carte}),\\
$\sqri C=\{x\in C:\,\cone(C-x) \ \mbox{is a closed linear subspace of} \ {\cal X} \};$
\item[$\bullet$] the \emph{quasi-relative interior} of $C$ (cf. \cite{Borwein-Lewis}),\\
$\qri C=\{x\in C:\cl\cone(C-x)\ \mbox{is a linear subspace of} \ {\cal X}\}$
\item[$\bullet$] the \emph{quasi interior} of $C$ (cf. \cite{Limber-Goodrich}),\\
$\qi C=\{x\in C:\cl\cone(C-x)={\cal X}\}.$
\end{enumerate}
For the last two notions we have the following dual characterizations.

\begin{proposition} \label{caracterizari int} (cf. \cite{Borwein-Lewis, qri}) Let $C$ be a nonempty convex subset of
${\cal X}$ and $x\in C$. Then:
\begin{enumerate}
\item[(i)] $x\in\qri C$ $\Longleftrightarrow$ $N_C(x)$ is a linear subspace of ${\cal X}^*$;
\item[(ii)] $x\in\qi C$ $\Longleftrightarrow$
$N_C(x)=\{0\}$.\end{enumerate}
\end{proposition}
For a convex set $C\subseteq {\cal X}$ one has the following
relations of inclusion for the generalized interiority notions introduced above:
\begin{equation*}
\inte C \subseteq\core C\subseteq
\begin{array}{c}
\sqri C\\
~\\
\qi C
\end{array}
\subseteq\qri C \subseteq C,
\end{equation*}
all of them being in general strict. Between $\sqri$ and $\qi$ no relation of inclusion holds in general. For a comprehensive discussion, examples and counterexamples with this respect we refer to \cite{qri1}. If ${\cal X}$ is a finite-dimensional space, then $\qi C =\inte C =\core C$
(cf.\cite{Limber-Goodrich}) and $\qri C =\sqri C =\ri C$ (cf.
\cite{Borwein-Lewis}), where $\ri C$ is the \emph{relative interior}
of $C$. In case $\inte C\neq\emptyset$ all the generalized
interiority notions collapse into the topological interior of the set $C$.

In the following we turn our attention to the \emph{Lagrange duality} for
the optimization problem with geometric and cone constraints
$$(P)\mbox{ }\mbox{ }\inf_{\substack{x\in S\\g(x)\in-K}}f(x).$$

Here ${\cal X}$ and ${\cal Z}$ are two separated locally convex spaces,
the latter being partially ordered by the nonempty convex cone
$K\subseteq {\cal Z},$ $S\subseteq {\cal X}$ is a nonempty set,
$f:{\cal X}\rightarrow\B$ is a proper function and
$g:{\cal X}\rightarrow {\cal Z}$ is a vector function fulfilling $\dom f\cap S\cap g^{-1}(-K)\neq\emptyset$. We
denote by $\geq_{K}$ the \emph{partial ordering} induced by $K$ on ${\cal Z}$, defined for $u,v \in {\cal Z}$ by $u \geq_K v$ whenever $u-v \in K$, and
by $K^*=\{x^*\in {\cal X}^*:\langle x^*,x\rangle \geq 0 \ \forall x\in K\}$
the \emph{dual cone} of $K$.

The $K$-\emph{epigraph} of $g:{\cal X}\rightarrow {\cal Z}$ is the set $\epi_{K}g=\{(x,z)\in {\cal X}\times {\cal Z}:z \geq_{K} g(x)\}$. The vector function $g$ is said to be $K$-\emph{convex} if $\epi_{K}g$ is convex and $K$-\emph{epi closed}  if $\epi_{K}g$ is closed.

We further assume that $S$ is a convex set, $f$ is a convex function and $g$ a $K$-convex vector function. The Lagrange dual problem associated to $(P)$ is
$$(D)\mbox{ }\mbox{ }\sup_{\lambda\in K^*}\inf_{x\in
S}\{f(x)+\langle\lambda,g(x)\rangle\}.$$ By $v(P)$ and $v(D)$ we denote the optimal objective values of the primal and
the dual problem, respectively. It is a known fact that between the
primal and the dual problem \emph{weak duality}, i.e. $v(P) \geq v(D)$, always holds.
In order to guarantee strong duality, i.e. the situation when $v(P) = v(D)$ and $(D)$ has an optimal solution,
we additionally need to require the fulfillment of a so-called \emph{qualification condition}. In the literature one can distinguish between two main classes
of qualification conditions, the so-called \emph{generalized interiority
point} and \emph{closedness-type} conditions, respectively. For an overview on the
relations between these two classes we refer to \cite{Bot}.

Throughout this paper we deal with qualification conditions of the first type and
discuss their applicability in the context of different topics involving convex risk measures.
To this end we consider the \emph{Slater constraint
qualification}
\begin{eqnarray*}\label{QC1} (QC_1)\mbox{  }\mbox{  }
\exists x'\in \dom f\cap S\mbox{
such that } g(x')\in -\inte K
\end{eqnarray*}
as well as the generalized interiority point qualification conditions (cf. \cite{Bot})
\begin{eqnarray*}\label{QC2} (QC_{2})&&  {\cal X} \mbox{ and } {\cal Z} \mbox{
are Fr\'echet spaces, } S \mbox{ is closed, } f \mbox{ is lower
semicontinuous, }\\ && g \mbox{ is $K$-epi closed and }
0\in\core(g(\dom f\cap S)+K),\end{eqnarray*}
and
\begin{eqnarray*}\label{QC3} (QC_{3})&&  {\cal X} \mbox{ and } {\cal Z} \mbox{
are Fr\'echet spaces, } S \mbox{ is closed, } f \mbox{ is lower
semicontinuous, }\\&& g \mbox{ is  $K$-epi closed and }
0\in\sqri(g(\dom f\cap S)+K).\end{eqnarray*}

Assuming that $v(P) \in \R$, along the above qualification conditions, we consider also the following one introduced in \cite{qri} (see, also, \cite{qrim, qri1}) and
expressed by means of the quasi interior and quasi-relative interior
\begin{eqnarray*}\label{QC4} (QC_{4}) &&  \exists x'\in \dom f\cap S\mbox{
such that } g(x')\in -\qri K,\  \cl(K-K)={\cal Z} \mbox{ and }\\
&&(0,0)\notin\qri\left[\co\left(\mathcal{E}_{v(P)}\cup\{(0,0)\}\right)\right]\end{eqnarray*}
where $\mathcal{E}_{v(P)}=\{(f(x)-v(P)+\epsilon,
g(x)+z):x\in\dom f\cap S, z\in K, \epsilon\geq 0\}$ is the set
in analogy to the \emph{the conic extension,} a notion used by
Giannessi in the theory of \emph{image spaces analysis}
(see \cite{giannessi}). If $0 \in \qi[(g(\dom f \cap S) + K) - (g(\dom f \cap S) + K)]$, then $(0,0)\notin\qri\left[\co\left(\mathcal{E}_{v(P)}\cup\{(0,0)\}\right)\right]$
is equivalent to $(0,0)\notin\qi\left[\co\left(\mathcal{E}_{v(P)}\cup\{(0,0)\}\right)\right]$. On the other hand, whenever $(P)$ has an optimal solution one has $\co\left(\mathcal{E}_{v(P)}\cup\{(0,0)\}\right) = \mathcal{E}_{v(P)}$. For further qualification conditions expressed by means of the quasi interior and quasi relative-interior we refer to \cite{qri,qri1}. Different to $(QC_i), i\in \{2,3\}$, these conditions have the remarkable property that they do not require the fulfillment of any topological assumption for the set $S$ or for the functions $f$ and $g$ and they do not restrict the spaces ${\cal X}$ and ${\cal Z}$ to be Fr\'echet. More than that, they find applicability in situations where $K$ is the ordering cone of a separable Banach space, like $\ell^p$ or $L^p$, $p \in [1,\infty)$ (see \cite{qrim,qri1,qri}). This is because of the fact that these ordering cones have nonempty quasi-relative interiors and quasi interiors, all the other interiority notions furnishing the empty set. The assumption that $v(P)$ is a real number is not restrictive at all, since, otherwise, namely, when $v(P) = -\infty$, strong duality is automatically fulfilled.

\begin{remark}\label{remqualcond} When ${\cal X}$ and ${\cal Z}$ are Fr\'echet spaces and $f,g$
are proper, convex and lower semicontinuous functions we have the
following relations between the above qualification conditions
$(QC_1)\Rightarrow(QC_2)\Rightarrow(QC_3)$ and, whenever $v(P) \in \R$,
$(QC_1)\Rightarrow(QC_2)\Rightarrow(QC_4)$. In
general the conditions $(QC_3)$ and $(QC_4)$ cannot be compared, for more on this topic
the reader being invited to consult \cite{qri1}.
\end{remark}

\begin{theorem}\label{strong_duality}
Assume that $v(P) \in \R$. If one of the qualification conditions $(QC_i), i\in \{1,...,4\}$, is fulfilled, then
$v(P)=v(D)$ and the dual problem has an optimal solution.
\end{theorem}

We consider further an atomless probability space $(\Omega,\mathfrak{F},\mathbb{P}),$ where $\Omega$ denotes the
space of future states $\omega,$ $\mathfrak{F}$ is a
\emph{$\sigma$-algebra} on $\Omega$ and $\mathbb{P}$ is a
\emph{probability measure} on $(\Omega,\mathfrak{F}).$ For a
measurable random variable $X:\Omega\to\R \cup \{+\infty\}$ the \emph{expectation
value} with respect to $\mathbb{P}$ is defined by
$\mathbb{E}(X):=\int_{\Omega}X(\omega)\,d\mathbb{P}(\omega).$ Whenever $X$ takes the value $+\infty$ on a subset of positive measure
we have $\mathbb{E}(X) = +\infty$. The \emph{essential supremum} of $X$, which represents the smallest essential upper bound of the random variable, is
$\essup X :=\inf\{a\in\R:\mathbb{P}\left(\omega:X(\omega)>a\right)=0\},$ while its \emph{essential infimum} is defined by
$\esinf X :=-\essup (-X)$.

Further, for $p\in[1,\infty)$ let we consider the following space of random variables
$$L^{p}:=L^{p}(\Omega,\mathfrak{F},\mathbb{P},\R)=\left\{X:\Omega\to\R:  X \mbox{ is measurable, }
\int_{\Omega}|X(\omega)|^{p}d\mathbb{P}(\omega)<+\infty\right\}.$$ The space
$L^{p}$ equipped with the norm
$\|X\|_{p}=\left(\mathbb{E}(|X|^{p})\right)^\frac{1}{p}$ is a Banach
space.  To complete the picture of
$L^{p}$ spaces, we introduce the space corresponding to the limiting
value $p=\infty,$ namely
$$L^{\infty}:=L^{\infty}(\Omega,\mathfrak{F},\mathbb{P},\R)=\left\{X:\Omega\to\R:
X \mbox{ is measurable, }\essup |X| <+\infty\right\},$$ which,
being equipped with the norm $\|X\|_{\infty}=\essup |X|$, is a Banach space, too.
For $p, q \in [1,\infty], p \geq q$, it holds $L^p \subseteq L^q$.
We denote the topological dual space of $L^p$ by $(L^{p})^{*}$ and for $p \in [1,\infty)$ one has that
$(L^{p})^{*}=L^q$, where $q\in(1,\infty]$ fulfills $q=p/(p-1)$ (with the convention $1/0 = \infty$).
In what concerns $(L^\infty)^*$, the topological dual space of $L^{\infty}$, this can be
identified with \emph{ba,} the space of all bounded finitely
additive measures on $(\Omega,\mathfrak{F})$ which are absolutely
continuous with respect to $\mathbb{P}$ and it is
usually much bigger than $L^{1}$. For the dual pairing $(X,X^*) \in (L^{p}, (L^{p})^{*})$ we shall write $\langle
X^{*},X\rangle=\mathbb{E}(X^{*}X)$ (even in the case $p=\infty$, by making an abuse of notation).

Equalities between random variables are to be interpreted in an \textit{almost everywhere (a.e.)} way, while for $X,Y \in L^p$ we write
$X \geq Y$ if and only if $X-Y \in L^{p}_{+}:=\{X\in L^{p}:X\geq 0 \ \mbox{a.e.}\}$. We also write $X > Y$ if $X(\omega) > Y(\omega)$ for almost every $\omega \in \Omega$. Random variables $X:\Omega\to\R$ which take a constant value $c\in\R,$ i.e $X=c$ a.e., will be identified with the real number $c$. Each random variable $X:\Omega\to\R$ can be represented as $X=X_{+}-X_{-}$, where $X_{+},X_{-}:\Omega\to\R$ are the random variables defined by $X_{+}(\omega)=\max\{0,X(\omega)\}$ and $X_{-}(\omega)=\max\{0,-X(\omega)\}$ for all $\omega \in \Omega$. The \emph{characteristic function} of a
set $G\in\mathfrak{F}$ is defined as being $\mathbf{1}_G:\Omega\to\R$,
$$\mathbf{1}_G(\omega)=\left\{
\begin{array}{ll}
1, & \mbox {if } \omega\in G,\\
0, & \mbox{otherwise}
\end{array}\right.$$
and, in view of the above notion, the expectation of a random variable
$X$ admits the equivalent representation $\mathbb{E}(X)=\langle\mathbf{1}_{\Omega}, X\rangle,$ which will be used several times in
this article.

Although the first axiomatic way of defining risk measures has been
given by Artzner, Delbaen, Eber and Heath in \cite{Artzen}
and refers to \emph{coherent risk measures}, it has become a standard
in modern risk management to assess the riskness of a portfolio by means of
\emph{convex risk measures}. The latter have been introduced by F\"{o}llmer
and Schied in \cite{Folmer}.

\begin{definition}\label{masura convexa}
We call \emph{risk function} a proper function $\rho:L^{p} \to \B, p \in [1,\infty]$. The risk function $\rho$ is said to be
\begin{enumerate}
\item[(i)] \mbox{\emph{convex}, if}: $\rho(\lambda X+(1-\lambda)Y)\leq\lambda\rho(X)+(1-\lambda)\rho(Y) \  \forall \lambda\in[0,1] \ \forall X,Y\in L^{p}$;
\item[(ii)] \mbox{\emph{positively homogeneous}, if}: $\rho(0) = 0$ and $\rho(\lambda X) = \lambda\rho(X) \  \forall \lambda > 0 \ \forall X \in L^{p}$;
\item[(iii)] \mbox{\emph{monotone}, if}: $X\geq Y$ $\Rightarrow \rho(X)\leq\rho(Y) \ \forall X,Y \in L^p$;
\item[(iv)] \mbox{\emph{cash-invariant}, if}: $\rho(X+a)=\rho(X)-a \ \forall X\in L^{p} \ \forall a\in\R$;
\item[(v)] \mbox{a \emph{convex risk measure}, if}: $\rho$ is convex, monotone and cash-invariant;
\item[(vi)] \mbox{a \emph{coherent risk measure}, if}: $\rho$ is a positively homogeneous convex risk measure.
\end{enumerate}
\end{definition}

The literature on convex risk measures has known in the last time a rapid growth. For examples of coherent and convex risk measures we refer the reader to
\cite{Bot-Lorenz, filipovic, Svindland, Folmer, Rock-cvar, Rock-cvar2, rock, Rusz}, some of them being the object of the investigations we make in the following sections of this paper. More precisely, we provide in the following subdifferential formulae and dual representations for different risk measures by making use of the conjugate duality theory in convex optimization, beyond the results obtained on this topic in \cite{Bot-Lorenz, Pflug, Rock-optim_cond, rock, Rusz}.

In the following section we consider a generalized convex risk measure defined via a so-called \emph{utility function} and associated with the Optimized Certainty Equivalent (OCE), a notion introduced and explored in \cite{teboulle1, teboulle2}. This convex risk measure is expressed as an infimal value function, thus we provide first of all a weak sufficient condition for the attainment of the infimum in its definition. Further, we give formulae for its conjugate function and its subdifferential. The generalized convex risk measure we consider has the advantage that, for some particular choices of the utility function, it leads to some well-known convex risk measures, for the conjugate and subdifferential of which we are consequently able to derive the corresponding formulae.

The results in the sections 4 and 5 have as starting point the paper of Filipovi\'c and Kupper \cite{filipovic}, where for a convex risk function the so-called \emph{monotone cash-invariant hull} has been introduced, which is actually the greatest monotone and cash-invariant function majorized by the risk function. This function has been formulated by making use of the \emph{infimal convolution}. In other words, the monotone cash-invariant hull at a given point is nothing else than the optimal objective value of a convex optimization problem. Having as a starting point this observation, we give here a dual representation of the monotone and cash-invariant hull by employing the Lagrange duality theory along with a qualification condition, under the hypothesis that the risk function is \emph{lower semicontinuous}. This guarantees the vanishing of the duality gap and, implicitly, the validity of the dual representation. The examples considered in \cite{filipovic} are discussed from this new point of view.

In the last section of the paper we deal with the same problem as in Section 4, but by considering this time a convex risk function which does not fulfill the lower semicontinuity assumption. For this function we can easily establish the \emph{monotone hull} and we can also give a dual representation for it by making use of the quasi-relative interiority-type qualification condition $(QC_4)$. We also refer to the limitations of this approach in the context of the determination of the \emph{monotone cash-invariant hull} for the function in discussion.

\section{Conjugate and subdifferential formulae for convex risk measures via Optimized Certainty Equivalent}

In this section we will furnish first formulae for both conjugate and subdifferential of a generalized convex risk measure, associated with the Optimized Certainty Equivalent (OCE). The Optimized Certainty Equivalent was introduced by Ben-Tal and Teboulle in \cite{teboulle1} by making use of a \emph{concave
utility function}. For properties of OCE and for relations with other certainty equivalent measures we refer to \cite{teboulle1,teboulle2}. For the investigations in this paper we adapt the definition of the Optimized Certainty Equivalent and the setting in which this has been introduced, by considering a \emph{convex utility function}, as this better suits in the general framework of convex duality. We close the section by particularizing the general results to some convex risk measures widely used in the literature.

Let us start by fixing the framework in which we work throughout the section.

\begin{assumption}\label{utility}
Let $v:\R\to \B$ be a proper, convex, lower semicontinuous and nonincreasing  function such that $v(0)=0$ and $ -1\in\partial v(0)$.
\end{assumption}

\begin{remark}\label{normalizare}
The two conditions we impose on the utility function $v$ are also known as \emph{normalization conditions}. By
exploiting the definition of the subdifferential, they can be equivalently written as $v(0) = 0$ and $v(t)+t \geq 0$ for all $t\in\R.$
\end{remark}

Having as starting point the definition of the Optimized Certainty Equivalent given in \cite{teboulle2} we define for $p \in [1,\infty]$ the following
\emph{generalized convex risk measure} $\rho_{v}:L^{p}\to \R \cup \{+\infty\}$
\begin{equation}\label{OCE}
\rho_{v}(X)= \inf_{\lambda\in\R}\{\lambda+\mathbb{E}(v(X+\lambda))\}.
\end{equation}

One can easily see that, due to the Assumption, $\rho_{v}(X) \geq -\mathbb{E}(X)$ for all $X \in L^p$ and that this function satisfies
the properties required in the definition of a convex risk measure. Next we provide a formula for the conjugate of $\rho_v$.

\begin{lemma}\label{conjug_oce}
The conjugate function of $\rho_{v}$, $\rho_{v}^{*}:(L^{p})^{*}\to\B$, is given by
\begin{eqnarray}\rho_{v}^{*}(X^*)=\left\{
\begin{array}{ll}
\mathbb{E}(v^{*}(X^{*})), & \mbox {if }\  \mathbb{E}(X^{*})=-1,\\
+\infty, & \mbox{otherwise}.
\end{array}\right.
\end{eqnarray}
\end{lemma}
\begin{proof} By the definition of the conjugate function we get for all $X^* \in (L^{p})^{*}$
 \begin{eqnarray*}
 \rho_{v}^{*}(X^{*})&=&\sup_{\substack{X\in L^{p}\\\lambda\in\R}}
 \{\langle X^*,X\rangle-\lambda-\mathbb{E}(v(X+\lambda))\}=\sup_{\substack{R\in L^{p}\\\lambda\in\R}}
 \{\langle X^*,R-\lambda\rangle-\lambda-\mathbb{E}(v(R))\}\\
 &=&\sup_{\substack{\lambda\in\R}}
\{-\lambda(\mathbb{E}(X^*)+1)\}+\sup_{\substack{R\in L^{p}}}
 \{\langle X^*,R\rangle-\mathbb{E}(v(R))\}.
\end{eqnarray*}
Using the interchangeability property of minimization and integration
(see, for instance, \cite[Theorem 14.60]{Rock-Wets}) the second expression from above can be written as
$$\sup_{\substack{R\in L^{p}}}  \{\langle X^*,R\rangle-\mathbb{E}(v(R))\} = \mathbb{E}\left \{\sup_{\substack{r\in\R}}(rX^{*}-v(r)) \right\} = \mathbb{E}(v^{*}(X^{*})).$$ On the other hand, since $\sup_{\lambda\in\R}\{-\lambda(\mathbb{E}(X^*)+1)\}=\delta_{\{0\}}(\mathbb{E}(X^*)+1)$, one obtains the desired conclusion.
\end{proof}

Before providing a subdifferential formula for $\rho_{v}$, we deliver via Lagrange duality a sufficient condition the utility function $v$ has to fulfill in order to guarantee the attainment of the infimum in the definition of $\rho_v(X)$ for all $X \in L^p$. According to \cite{teboulle1,teboulle2}, for those $X \in L^p$ having as support a bounded and closed interval, the infimum in \eqref{OCE} is attained. But what we provide here, is a condition which  ensures this fact independently from the choice of the random variable.

Let $X \in L^p$ be fixed. Consider the following primal optimization problem
\begin{equation}\label{primala}
\inf_{\substack{\Xi \in L^q\\ \mathbb{E}(\Xi)=-1}}  \bigg[\mathbb{E}(v^{*}(\Xi))-\langle X, \Xi \rangle  \bigg],
\end{equation}
where $q:=\frac{p}{p-1}$, if $p \in [1,\infty)$, and $q:=1$, if $p=\infty$. The \emph{Lagrange dual} optimization problem to
\eqref{primala} is given by
\begin{equation*}
\sup_{\lambda \in \R} \inf_{\Xi \in L^q} \Big [ \mathbb{E}(v^{*}(\Xi)) -\langle X, \Xi \rangle + \lambda (\mathbb{E}(\Xi) + 1) \Big] = \sup_{\lambda \in \R} \bigg [\lambda - \sup_{\Xi \in L^q} \Big (\langle X-\lambda, \Xi \rangle - \mathbb{E}(v^{*}(\Xi)) \Big)\bigg ].
\end{equation*}
Again, via \cite[Theorem 14.60]{Rock-Wets}, it holds
$$\sup_{\Xi \in L^q}\bigg(\langle X-\lambda, \Xi \rangle - \mathbb{E}(v^{*}(\Xi))\bigg) = \mathbb{E} \bigg(\sup_{\substack{r\in\R}}(r(X-\lambda)-v^*(r))\bigg) = \mathbb{E}(v(X-\lambda))$$
and this leads to the following dual problem to \eqref{primala}
\begin{equation}\label{duala}
\sup_{\substack{\lambda\in\R}} \bigg[-\lambda - \mathbb{E}(v(X + \lambda))\bigg].
\end{equation}
Let us notice that the optimal objective value of the dual problem \eqref{duala} is equal to $-\rho_v(X)$.

\begin{theorem}\label{thexist}
Assume that
\begin{equation}\label{cond}
\{d\in\R:v_{\infty}(d)=-d\}=\{0\}.
\end{equation}
Then for all $X \in L^p$ there exists $\bar \lambda(X) \in \R$ such that $\rho_v(X) = \bar \lambda(X) +\mathbb{E}(v(X+\bar \lambda(X)))$.
\end{theorem}
\begin{proof}
We consider $X \in L^p$ fixed and prove that under condition \eqref{cond} for the primal-dual pair \eqref{primala}-\eqref{duala} strong duality holds. This will guarantee among others the existence of an optimal solution $\bar \lambda(X)$ for the dual, which will prove the assertion.

Define $s:\R\to\B$ as being $s(t) = v(t) + t$. Notice that $s$ is proper, convex and lower semicontinuous, too, and
for all $t^* \in \R$ it holds $s^{*}(t^{*})=v^{*}(t^{*}-1)$, so $\dom s^{*}=\dom v^{*}+1$. On the other hand, since $0 \in \dom s$, it holds (see \eqref{asymptotic_func}) \begin{equation}\label{condinf}
s_{\infty}(d)=\sup_{t>0} \frac{s(td)-s(0)}{t}= \sup_{t>0} \frac{v(td)+td}{t} = \sup_{t>0} \frac{v(td)-v(0)}{t} + d = v_\infty(d) + d \ \forall d \in \R.
\end{equation}
Thus condition \eqref{cond} is nothing else than asking that $\{d\in\R:s_{\infty}(d)=0\}=\{0\}$. On the other hand, from \eqref{condinf} it follows that $s_{\infty}(d) \geq 0$ for all $d \in \R$. By taking into account \cite[Theorem 3.2.1.]{teboulle-carte} we get that $0\in\ri(\dom s^{*})=\ri(\dom v^{*}+1)$.

Further, we notice that, by taking $f:L^q \rightarrow \overline{\R}$, $f(\Xi) = \mathbb{E}(v^{*}(\Xi))-\langle X, \Xi \rangle$ and $g:L^q \rightarrow \R$,
$g(\Xi) = \mathbb{E}(\Xi)+1$, which are both convex functions, the qualification condition $(QC_3)$ is fulfilled. Indeed, $f$ is lower semicontinuous, $g$ is continuous and $0\in \ri(\dom v^{*}+1) = \sqri(\mathbb{E}(\dom v^{*})+1)$. Thus the existence of strong duality for \eqref{primala}-\eqref{duala} and, consequently, of an optimal solution for \eqref{duala} is shown.
\end{proof}

Next we provide a formula for the subdifferential of the general convex risk measure $\rho_v$.

\begin{theorem}\label{thsubdiff}
Assume that condition \eqref{cond} is fulfilled. Let $X \in L^p$ and $\bar \lambda(X) \in \R$ be the element where the infimum in the definition of
$\rho_v(X)$ is attained. Then it holds
\begin{equation}\partial\rho_{v}(X)=\{X^{*}\in (L^{p})^{*}: X^{*}(\omega) \in\partial v(X(\omega) + \bar \lambda(X)) \ \mbox{for almost every} \ \omega \in \Omega, \mathbb{E}(X^{*})=-1\}.\label{subdiff}
\end{equation}
\end{theorem}

\begin{proof} We fix an $X \in L^p$ and let $\bar \lambda(X) \in \R$ be such that $\rho_v(X) = \bar \lambda(X) +\mathbb{E}(v(X+\bar \lambda(X)))$. Then, via \eqref{subdiff-conjug-zero},
$$X^* \in \partial\rho_{v}(X) \Leftrightarrow \rho_{v}^*(X^*) + \rho_{v}(X) = \langle X^*,X\rangle$$
or, equivalently, (see Lemma \ref{conjug_oce})
$$\mathbb{E}(X^{*})=-1 \ \mbox{and} \ \mathbb{E}(v^{*}(X^{*}) + v(X+\bar \lambda(X)) - \langle X^*,X + \bar \lambda(X)\rangle) = 0.$$
On the other hand, by the Young-Fenchel inequality, it holds
$$v^*(X^*(\omega)) + v(X(\omega) + \bar \lambda(X)) -  X^*(\omega)(X(\omega) + \bar \lambda(X)) \geq 0 \ \forall \omega \in \Omega,$$
which means that $\mathbb{E}(v^{*}(X^{*}) + v(X+\bar \lambda(X)) - \langle X^*,X + \bar \lambda(X)\rangle) = 0$ is nothing else than
$v^*(X^*(\omega)) + v(X(\omega) + \bar \lambda(X)) -  X^*(\omega)(X(\omega) + \bar \lambda(X)) = 0$ for almost every $\omega \in \Omega$. In
conclusion, $X^* \in \partial\rho_{v}(X)$ if and only if
$$\mathbb{E}(X^{*})=-1 \ \mbox{and} \ X^{*}(\omega) \in\partial v(X(\omega) + \bar \lambda(X)) \ \mbox {for almost every} \ \omega \in \Omega.$$
\end{proof}

In the sequel we rediscover for particular choices of the utility function $v$ several well-known convex risk measures and provide formulae for their conjugates and subdifferentials.

\subsection{Conditional value-at-risk (CVaR)}\label{sub21}

For $\gamma_{2}<-1<\gamma_{1}\leq 0$ we consider the utility function $v_{1}:\R\to\R$ defined by
$$v_{1}(t)=\left\{
\begin{array}{ll}
\gamma_{2}t, & \mbox {if } t\leq 0,\\
\gamma_{1}t, & \mbox {if } t> 0,
\end{array}\right.$$
and notice that it satisfies all the requirements in the Assumption. This gives rise to the following convex risk measure $\rho_{v_1} : L^p \rightarrow \R$,
\begin{equation*}
\rho_{v_{1}}(X) = \inf_{\substack{\lambda\in\R}}\{\lambda+\gamma_{1}\mathbb{E}(X+\lambda)_{+}-\gamma_{2}\mathbb{E}(X+\lambda)_{-}\}.
\end{equation*}
Since $v_1^* = \delta_{[\gamma_2,\gamma_1]}$, via Lemma \ref{conjug_oce} one gets for $\rho_{v_{1}}^* : (L^p)^* \rightarrow \B$ the following expression
\begin{equation*}
\rho_{v_1}^{*}(X^*)=\left\{
\begin{array}{ll}
0, & \mbox {if} \  \gamma_2 \leq X^* \leq \gamma_1, \mathbb{E}(X^{*})=-1,\\
+\infty, & \mbox{otherwise}.
\end{array}\right.
\end{equation*}
Noticing that for all $d \in \R$,
\begin{equation*}
(v_1)_\infty(d)=\left\{
\begin{array}{ll}
\gamma_2d, & \mbox {if} \  d < 0,\\
0, & \mbox {if} \  d = 0,\\
\gamma_1d, & \mbox {if} \  d > 0,
\end{array}\right.
\end{equation*}
one can easily see that condition \eqref{cond} is satisfied. Thus for all $X \in L^p$ there exists $\bar \lambda(X) \in \R$ such that
$\rho_{v_{1}}(X) = \bar \lambda(X) +\gamma_{1}\mathbb{E}(X+ \bar \lambda(X))_{+}-\gamma_{2}\mathbb{E}(X+ \bar \lambda(X))_{-}$. Further, according to Theorem \ref{thsubdiff}, we will make use of $\bar \lambda(X)$ when giving the formula for the subdifferential of $\rho_{v_{1}}$ at $X$. Since
$$
\partial v_{1}(t)=\left \{
\begin{array}{ll}
\{\gamma_{2}\}, & \mbox{if} \ t<0,\\{}
[\gamma_2, \gamma_1], & \mbox{if} \ t=0,\\
\{\gamma_{1}\}, & \mbox{if} \ t>0,\\
\end{array} \right.
$$
via \eqref{subdiff} we obtain for all $X \in L^p$ the following formula
\begin{equation}\label{subdifg}
\partial \rho_{v_1}(X)=\left\{X^{*}\in (L^{p})^{*}: \mathbb{E}(X^{*})=-1,
\begin{array}{ll}
X^*(\omega) = \gamma_2, & \mbox {if } X(\omega)<-\bar \lambda(X),\\
X^*(\omega) \in [\gamma_2,\gamma_1], & \mbox {if } X(\omega)=-\bar \lambda(X),\\
X^*(\omega) = \gamma_1, & \mbox {if } X(\omega)>-\bar \lambda(X)
\end{array}\right\}.
\end{equation}

When $\gamma_{1}=0$ and $\gamma_{2}=-1/\beta,$ where $\beta\in(0,1)$, the convex risk measure $\rho_{v_1}$ turns out to be the classical
so-called \emph{conditional value-at-risk} (see, for instance, \cite{Rock-cvar,Rock-cvar2}), $\cvar_\beta : L^p \rightarrow \R$,
\begin{equation}
\cvar\nolimits_{\beta}(X)=\inf_{\substack{\lambda\in\R}}\left \{\lambda+\frac{1}{\beta}\mathbb{E}[(X+\lambda)_{-}] \right\}.\label{Min-cvar}
\end{equation}

Thus, for all $X^* \in (L^p)^*$ its conjugate function $\cvar_\beta^* : (L^p)^* \rightarrow \B$ looks like
\begin{equation*}
\cvar\nolimits_{\beta}^{*}(X^*)=\left\{
\begin{array}{ll}
0, & \mbox {if} \  -\frac{1}{\beta} \leq X^* \leq 0, \mathbb{E}(X^{*})=-1,\\
+\infty, & \mbox{otherwise}.
\end{array}\right.
\end{equation*}
For all $X \in L^p$ the element where the infimum in the definition of $\cvar\nolimits_{\beta}(X)$ is attained, is the so-called \emph{value-at-risk of $X$ at level $\beta$},
$$\var\nolimits_\beta(X)=-\inf\{\alpha : \mathbb{P}(X \leq \alpha) > \beta\}.$$
This fact, along with \eqref{subdifg}, furnishes for all $X \in L^p$ the following formula for the subdifferential of the conditional value-at-risk
\begin{equation*}
\partial \cvar\nolimits_{\beta}(X)=\left\{\!X^{*}\in (L^{p})^{*}\!: \!\mathbb{E}(X^{*})=-1,\!\!\!
\begin{array}{ll}
X^*(\omega) = -1/\beta, & \mbox {if } X(\omega)<-\var_\beta(X),\\
X^*(\omega) \in [-1/\beta,0], & \mbox {if } X(\omega)=-\var_\beta(X),\\
X^*(\omega) = 0, & \mbox {if } X(\omega)>-\var_\beta(X)
\end{array}\!\!\!\!\right\}.
\end{equation*}

For alternative approaches for deriving the formula of the subdifferential of the conditional value-at-risk we refer to \cite{Rock-optim_cond, Rusz}.

\subsection{Entropic risk measure}\label{sub22}

Consider the utility function $v_{2}:\R\to\R$, $v_{2}(t)=\exp(-t)-1$, which obviously fulfills the hypotheses in the Assumption.
The convex risk measure we define via $v_2$ is $\rho_{v_{2}} : L^p \rightarrow \R$,
$\rho_{v_{2}}(X)=\inf_{\substack{\lambda\in\R}}\{\lambda+\mathbb{E}(\exp(-X-\lambda)-1)\}$. With the convention $0 \ln(0) = 0$ we have
for all $t^* \in \R$ that
\begin{equation*}
v_{2}^{*}(t^{*})=\left\{
\begin{array}{ll}
-t^*\ln(-t^{*})+t^{*}+1, & \mbox {if } t^{*} \leq 0,\\
+\infty, & \mbox {if } t^{*}>0,
\end{array}\right.
\end{equation*}
and, so, from Lemma \ref{conjug_oce} it follows that for all $X^* \in (L^p)^*$ one has
\begin{equation*}
\rho_{v_{2}}^{*}(X^{*})=\left\{
\begin{array}{ll}
-\mathbb{E}(X^*\ln(-X^{*})), & \mbox {if } X^{*}<0,\  \mathbb{E}(X^{*})=-1,\\
+\infty, & \mbox {otherwise}.
\end{array}\right.
\end{equation*}
Since $(v_2)_\infty = \delta_{[0,+\infty)}$, condition \eqref{cond} is fulfilled and for all $X \in L^p$ there exists $\bar \lambda (X) \in \R$ such that
the infimum in the definition of $\rho_{v_{2}}(X)$ is  attained at this point. But in this special case one can easily see that $\bar \lambda (X) = \ln(\mathbb{E}(\exp(-X))$ and therefore the risk measure can be equivalently written as $\rho_{v_{2}}(X)=\ln(\mathbb{E}(\exp(-X)))$. This is the so-called
\emph{entropic risk measure} introduced and investigated in \cite{Barrieu}.

Noticing that $\partial v_{2}(t)=\{\nabla v_{2}(t)\}=\{-\exp(-t)\}$ for all $t \in \R$, the subdifferential of the entropic risk measure at $X \in L^p$ is
$\partial\rho_{v_{2}}(X) = \{\nabla \rho_{ v_{2}}(X)\} =  \left\{\frac{-1}{\mathbb{E}(\exp(-X))}\exp(-X)\right\}$.

\subsection{The worst-case risk measure}\label{sub23}

By taking as utility function $v_3 = \delta_{[0, +\infty)}$ one rediscovers under $\rho_{v_3} : L^p \rightarrow \R \cup \{+\infty\}$,
\begin{equation}\rho_{v_{3}}(X)=\inf_{\substack{\lambda\in\R\\ X+\lambda \geq0}}\lambda =-\esinf X,\label{esinf}
\end{equation}
the so-called \emph{worst-case risk measure}. As $v_3^* = \delta_{(-\infty,0]}$, we have for all $X^* \in (L^p)^*$ that
$$\rho_{v_3}^{*}(X^{*})=\left\{
\begin{array}{ll}
0, & \mbox {if } X^{*}\leq0,\ \mathbb{E}(X^{*})=-1,\\
+\infty, & \mbox{otherwise}.
\end{array}\right.
$$
Noticing that $(v_3)_\infty =  \delta_{[0, +\infty)}$, one can easily see that \eqref{cond} is fulfilled, which means that for all
$X \in L^p$ there exists $\bar \lambda(X) \in \R$ at which the infimum in \eqref{esinf} is attained. If $\esinf X = -\infty$, then
one can take $\bar \lambda(X)$ arbitrarily in $\R$, while, when $\esinf X \in \R$, $\bar \lambda(X) = -\esinf X$. Since
$$
\partial v_{3}(t)=\left \{
\begin{array}{ll}
\emptyset, & \mbox{if} \ t<0,\\{}
(-\infty,0], & \mbox{if} \ t=0,\\
\{0\}, & \mbox{if} \ t>0,\\
\end{array} \right.
$$
we can provide via Theorem \ref{thsubdiff} the formula for the subdifferential of the worst-case risk measure. Indeed, for
$X \in L^p$ with $\esinf X = -\infty$ one has $\partial \rho_{v_3}(X) = \emptyset$, while, if $\esinf X \in \R$, it holds
\begin{equation*}
\partial\rho_{v_{3}}(X)=\left\{X^{*}\in (L^{p})^{*}: \mathbb{E}(X^{*})=-1,
\begin{array}{ll}
X^*(\omega) \in (-\infty,0], & \mbox {if } X(\omega) = \esinf X,\\
X^*(\omega) = 0, & \mbox {if } X(\omega) > \esinf X
\end{array}
\right\}.
\end{equation*}

\section{Dual representations of\! monotone and \!cash-invariant hulls}\label{sec3}

Throughout the economical literature one finds a vast variety of risk
functions, along the coherent and convex ones some very irregular ones, which are
neither monotone nor cash-invariant being present, too. In order to overcome the lack
of monotonicity or cash-invariance and to provide better tools for
quantifying risk, Filipovi\'{c} and Kupper have proposed in \cite{filipovic} the notions of \emph{monotone} and \emph{cash-invariant hulls},
which are the greatest monotone and, respectively, cash-invariant functions majorized by the risk function in discussion.
For the majority of the examples treated in \cite{filipovic} these hulls are not given in their initial formulation, but tacitly some dual representations
of them are used.

In this section we show that these dual representations are nothing else than the dual problems of the primal optimization problems hidden in the definition of the monotone and cash-invariant hulls and formulate sufficient qualification conditions for the existence of strong duality. This is the premise for making the dual representations viable. Finally, we discuss the examples from \cite{filipovic} and show that for those particular situations the qualification conditions are automatically fulfilled, fact which permits the formulation of \emph{refined} dual representations.

For the beginning we work in the general setting of a separated locally convex vector space ${\cal X}$ with ${\cal X}^*$ its
topological dual space. Further, let $\mathcal{P}$ be a nonempty convex closed cone in ${\cal X}$, $\Pi \in {\cal X} \setminus \{0\}$ and $f : {\cal X} \rightarrow \B$ a proper function. The following notions have been introduced in \cite{filipovic} having as a starting point the corresponding ones in the
definition of a convex risk measure.

\begin{definition}\label{defsec41}
The function $f$ is called:
\begin{enumerate}
\item[(i)] $\mathcal{P}$-\mbox{\emph{monotone}, if}: $x\geq_{\mathcal{P}}y \Rightarrow f(x) \leq f(y) \ \forall x,y \in {\cal X}$;
\item[(ii)] $\Pi$-\mbox{\emph{invariant}, if}: $f(x+a\Pi)=f(x)-a \ \forall x \in {\cal X} \ \forall a \in \R$.
\end{enumerate}
\end{definition}
If ${\cal X}=L^p$, $\mathcal{P} = L^p_+$ and $\Pi=1$, then one rediscovers in the definition above the monotonicity and cash-invariance, respectively,
as introduced in Definition \ref{masura convexa}.

Before introducing the following notions we consider the set $\mathcal{D}:=\{x^{*}\in {\cal X}^{*}:\langle x^{*},\Pi\rangle=-1\}$ and notice that for
the conjugate of the indicator function of $\mathcal{D}$ we have (see, for instance, \cite[Lemma 3.3]{filipovic}) for all $x \in {\cal X}$ that
$$\delta^{*}_{\mathcal{D}}(x) = \sup_{x^* \in \mathcal{D}} \langle x^*,x \rangle = \left \{\begin{array}{ll}
-a, & \mbox{if} \ \exists a \in \R \ \mbox{such that} \ x=a\Pi,\\
+\infty, & \mbox{otherwise}.
\end{array} \right.$$
This means that $\dom \delta^{*}_{\mathcal{D}} = \R\Pi:=\cup_{a \in \R} a\Pi$.

\begin{definition}\label{defsec42}
For the given function $f$ we call
\begin{enumerate}
\item[(i)] \emph{$\mathcal{P}$-monotone hull} of $f$ the function $f_{\mathcal{P}}:{\cal X} \rightarrow \B$ defined as
$$f_{\mathcal{P}}(x):=f\square\delta_{\mathcal{P}}(x)=\inf\{f(y): y \in {\cal X}, x\geq_{\mathcal{P}}y\}$$

\item[(ii)]\emph{$\Pi$-invariant hull} of $f$ the function $f_{\Pi}:{\cal X} \rightarrow \B$ defined as
$$f_{\Pi}(x):=f \square \delta^{*}_{\mathcal{D}}(x) = \inf_{a\in\R}\{f(x-a\Pi)-a\}.$$

\item[(iii)]\emph{$\mathcal{P}$-monotone $\Pi$-invariant hull} of $f$ the function $f_{\mathcal{P},\Pi} : {\cal X} \rightarrow \B$ defined as
$$f_{\mathcal{P},\Pi}(x):=f\square\delta_{\mathcal{P}}\square\delta^{*}_{\mathcal{D}}(x) = \inf \{f(y)-a: y \in {\cal X}, a \in \R, x\geq_{\mathcal{P}} y+ a\Pi\}.$$
\end{enumerate}
\end{definition}
Obviously, $\dom f_{\mathcal{P}} = \dom f + {\mathcal{P}}$, $\dom f_{\Pi} = \dom f + \R\Pi$ and $\dom f_{\mathcal{P},\Pi} =
\dom f + {\mathcal{P}} + \R\Pi$. Moreover, $f$ is ${\mathcal{P}}$-monotone if and only if $f = f_{\mathcal{P}}$, while
$f$ is $\Pi$-invariant if and only if $f = f_\Pi$.

In the following we assume that $f$ is a proper and convex function and provide a dual representation for $f_{\mathcal{P},\Pi}$ by making use of the convex duality theory. This approach is based on the observation that the value of the $\mathcal{P}$-monotone $\Pi$-invariant hull at a given point is nothing else than the optimal objective value of a convex optimization problem. Let $x \in \dom f + {\mathcal{P}} + \R\Pi$ be fixed and consider the following primal problem having as optimal objective value $f_{\mathcal{P},\Pi}(x)$
\begin{equation}
\inf_{\substack{y \in {\cal X}, a \in \R\\y+a\Pi - x \in -\mathcal{P}}} f(y)-a.\label{primalaPPi}\end{equation}
Its Lagrange dual problem looks like
$$\sup_{\substack{x^{*}\in \mathcal{P}^{*}}}\inf_{\substack{y\in {\cal X}, a \in \R}} \{f(y)-a+\langle x^{*},y+a\Pi-x\rangle\}$$
or, equivalently,
$$\sup_{\substack{x^{*}\in \mathcal{P}^{*}}}\left\{\inf_{\substack{a\in\R}} \{a(\langle x^{*},\Pi\rangle-1)\}+\inf_{\substack{y\in {\cal X}}}\{\langle x^{*},y\rangle+f(y)\}-\langle x^{*},x\rangle\right\}.$$
Since $\inf_{a \in\R}\{a(\langle x^{*},\Pi\rangle-1)\}=-\delta_{\{0\}}(\langle x^{*},\Pi\rangle-1)$, the dual
problem becomes
\begin{equation}
\sup_{\substack{x^{*}\in -\mathcal{P}^{*}\\ \langle x^{*},\Pi\rangle=-1}}\{\langle x^{*},x\rangle - f^{*}(x^{*})\}.
\label{dualaP-PI}
\end{equation}
This means that if one is able to guarantee strong duality for the primal-dual pair \eqref{primalaPPi}-\eqref{dualaP-PI}, then one has
\begin{equation}\label{dualrep}
f_{\mathcal{P},\Pi}(x) = \max_{\substack{x^{*}\in -\mathcal{P}^{*}\\ \langle x^{*},\Pi\rangle=-1}}\{\langle x^{*},x\rangle - f^{*}(x^{*})\},
\end{equation}
where by the use of $\max$ instead of $\sup$ we signalize the fact that the supremum is attained.

\begin{remark}\label{remppi}
If $f$ is $\mathcal{P}$-monotone, then $f = f\square\delta_{\mathcal{P}}$, which means that $f^* = f^* + \delta_{-\mathcal{P}^{*}}$. In this situation one would get for
$f_{\mathcal{P},\Pi}(x) = f_{\Pi}(x)$ the following dual representation
\begin{equation}
\sup_{\substack{\langle x^{*},\Pi\rangle=-1}}\{\langle x^{*},x\rangle - f^{*}(x^{*})\}.
\label{dualaP-P}
\end{equation}
On the other hand, if $f$ is $\Pi$-invariant, then $f = f \square \delta^{*}_{\mathcal{D}}$, which means that $f^* = f^* + \delta_{\mathcal{D}}$. In this situation one would get for
$f_{\mathcal{P},\Pi}(x) = f_{\mathcal{P}}(x)$ the following dual representation
\begin{equation}
\sup_{\substack{x^{*}\in -\mathcal{P}^{*}}}\{\langle x^{*},x\rangle - f^{*}(x^{*})\}.
\label{dualaP-PPI}
\end{equation}

\end{remark}

Next we investigate and discuss the existence of strong duality for the primal-dual pair \eqref{primalaPPi}-\eqref{dualaP-PI}. Since $g:{\cal X}\times\R\to {\cal X}$,
$g(y,a)=y+a\Pi-x,$ is an affine and continuous function,  one could try to guarantee to this aim that one of the qualification conditions
(see $(QC_1)$)
\begin{equation}\label{qcsec41}
\exists (y',a') \in \dom f \times \R \ \mbox{such that} \ y'+a'\Pi-x \in -\inte \mathcal{P}
\end{equation}
and (see $(QC_3)$)
\begin{equation}\label{qcsec42}
{\cal X} \mbox{ is} \ \mbox{a Fr\'echet space}, f \mbox{ is lower
semicontinuous and } x \in\sqri(\dom f + \R\Pi + \mathcal{P})
\end{equation}
is fulfilled. The qualification condition of quasi-relative interior-type $(QC_4)$ will be studied in the next section.

In the following we investigate the verifiability of these qualification conditions in the context of risk measure theory, namely by assuming that ${\cal X}=L^p$ and $\mathcal{P}=L^p_+$ for $p \in [1,\infty]$.  Working in this framework, one can easily see that the qualification condition
\eqref{qcsec41} is for $p \in [1,\infty)$ not verified, $L^p_+$ having an empty interior; therefore we will concentrate ourselves first on condition \eqref{qcsec42} and assume to this end that \emph{$f$ is lower semicontinuous}. A situation when $f$ fails to have this topological property will be addressed in the next section.

Thus, in order to guarantee the dual representation \eqref{dualrep} for $X \in \dom f + L^p_+ + \R\Pi$, one could ensure that
\begin{equation}\label{qcsecsp40}
X \in \sqri(\dom f + \R\Pi + L^p_+).
\end{equation}
This is the case when
\begin{equation}\label{qcsecsp41}
-L^p_+ \subseteq \dom f,
\end{equation}
but also when
\begin{equation}\label{qcsecsp42}
p=\infty \ \mbox{and} \ \esinf \Pi \cdot \essup \Pi >0.
\end{equation}

Indeed, when \eqref{qcsecsp41} holds, then $\dom f + \R\Pi + L^p_+ = L^p$ and \eqref{qcsecsp40} is valid. Suppose now that the second condition is true, namely $p=\infty$ and $\esinf \Pi \cdot \essup \Pi > 0$. In this situation the equality $\dom f + \R\Pi + L^\infty_+ = L^\infty$ holds, too. Assume that either $\esinf \Pi > 0$ or $\essup \Pi < 0$ and take an arbitrary $Z \in L^\infty$. Let be $Y \in \dom f$. Then there exists $a \in \R$ such that $(Z-Y) + a\Pi \in L^\infty_+$ and so $Z \in \dom f + \R\Pi + L^\infty_+$. Thus \eqref{qcsecsp42} is another sufficient condition for \eqref{qcsecsp40}. One can notice that the assumption made on $\Pi$ in the second condition is fulfilled when $\Pi \in L^\infty$ is a \emph{constant numeraire}.

It is also worth mentioning that condition \eqref{qcsecsp42} is sufficient for \eqref{dualrep}, even without assuming lower semicontinuity for $f$, since in the singular case $p=\infty$ the ordering cone has a nonempty interior. As $X-Y-a\Pi \in L^\infty_+$ for some $Y \in \dom f$ and $a \in \R$, under \eqref{qcsecsp42}, one can guarantee the existence of $a' \in \R$ such that $X-Y-a'\Pi \in \inte L^\infty_+ = \{Z \in L^\infty: \esinf Z > 0\}$. Thus, via \eqref{qcsec41}, the dual representation for the $\mathcal{P}$-monotone $\Pi$-invariant hull of $f$ is valid.

In the last part of this section we discuss the examples treated in \cite{filipovic} from this new perspective given by the duality theory,
investigate the fulfillment of the conditions \eqref{qcsecsp41} and \eqref{qcsecsp42} and
provide some refined dual representations for the risk functions in discussion. We will use the notion \emph{monotone} for \emph{$L^p_+$-monotone} and
\emph{cash-invariant} for \emph{$1$-invariant}. The same applies when we speak about the corresponding hulls.

\begin{example}\label{ex41}
For $p \in [1,\infty)$ and $c > 0$ consider the \emph{$L^{p}$ deviation risk measure} $f:L^{p}\to\R$ defined by
$f(X)=c\|X-\mathbb{E}(X)\|_{p}-\mathbb{E}(X)$. This is a convex, continuous and cash-invariant ($\Pi =1$) risk function, but not monotone in general.
For the conjugate formula of the $L^{p}$ deviation risk measure we refer to \cite{Bot-Lorenz}. This is for $X^* \in L^q$ given by
$$f^{*}(X^*)=\left\{
\begin{array}{ll}
0, & \mbox {if } \exists Y^{*}\in L^{q} \ \mbox{such that} \ c( Y^{*}-\mathbb{E}(Y^{*}))-1=X^{*},\ \|Y^{*}\|_{q}\leq1,\\
+\infty, & \mbox{otherwise}.
\end{array}\right.$$
As $\dom f = L^p$, \eqref{qcsecsp41} is valid and thus the \emph{monotone hull} of $f$ looks for all $X \in L^p$ like (see also Remark \ref{remppi})
\begin{equation*}
f_{L^{p}_{+},1}(X) = f_{L^{p}_{+}}(X)=\max_{\substack{\|Y^*\|_{q}\leq 1\\ c(Y^{*}-\mathbb{E}(Y^{*})) \leq 1}} c[\mathbb{E} (Y^{*})\mathbb{E}(X)-\mathbb{E} (Y^{*}X)] - \mathbb{E}(X).
\end{equation*}
In this way we rediscover the formula given in \cite[Subsection 5.1]{filipovic}.
\end{example}

\begin{example}\label{ex42}
Closely related to previous example we consider for $p \in [1,\infty)$ and $c > 1$ the
\emph{$L^{p}$ semi-deviation risk measure} $f:L^{p}\to\R$ defined
as $f(X)=c\|(X-\mathbb{E}(X))_{-}\|_{p}-\mathbb{E}(X)$. This is a convex, continuous and cash-invariant ($\Pi =1$) risk function, but not monotone in general.
For its conjugate function we have for $X^* \in L^q$ the following formula (see \cite{Bot-Lorenz})
$$f^{*}(X^*)=\left\{
\begin{array}{ll}
0, & \mbox {if } \exists Y^{*}\in -L^{q}_{+} \ \mbox{such that} \  c(Y^{*}-\mathbb{E}(Y^{*}))-1=X^{*},\ \|Y^{*}\|_{q}\leq1,\\
+\infty, & \mbox{otherwise}.
\end{array}\right.$$
Consequently, since \eqref{qcsecsp41} is valid, the \emph{monotone hull} of $f$ is for all $X \in L^p$ given by (see also \cite[Subsection 5.2]{filipovic})
\begin{equation*}
f_{L^{p}_{+},1}(X) = f_{L^{p}_{+}}(X)=\max_{\substack{ Y^* \in L^q_+, \|Y^*\|_{q}\leq 1\\ c(Y^{*}-\mathbb{E}(Y^{*})) \leq 1}} c[\mathbb{E} (Y^{*})\mathbb{E}(X)-\mathbb{E} (Y^{*}X)] - \mathbb{E}(X).
\end{equation*}
\end{example}

\begin{example}\label{ex43}
For $p\in[1,\infty)$ and $c > 0$ consider the \emph{mean-$L^p$ risk measure}
$f:L^{p}\to\R$ defined as $f(X)=c/p\mathbb{E}(|X|^{p})-\mathbb{E}(X) = c/p \|X\|_p^p - \mathbb{E}(X)$, which is a convex and continuous
risk function but neither monotone nor cash-invariant ($\Pi=1$). Its conjugate function can be easily derived from
\cite{Bot-Lorenz} and for $X^* \in L^q$ it looks like
$$f^{*}(X^{*})=\frac{p-1}{pc^\frac{1}{p-1}}\mathbb{E}(|X^*+1|^q).$$
Again, $\dom f = L^p$, which means that the \emph{monotone cash-invariant hull} of $f$ has for all $X \in L^p$ the following formulation
\begin{equation*}
f_{L^{p}_+,1}(X)=\max_{\substack{X^{*}\in
-L^{q}_{+}\\\mathbb{E}(X^*)=-1}}\mathbb{E}\left[X^*X -\frac{1}{c^{q-1}q} |X^*+1|^q\right].
\end{equation*}
Different to the approach in \cite[Subsection 5.3]{filipovic}, the use of the strong duality theory allows us to guarantee the attainment of the supremum in the formula above.
\end{example}

\begin{example}\label{ex44}
For $p\in[1,\infty)$ and $c>0$ consider now the \emph{$L^{p}$ semi-moment risk measure} $f:L^{p}\to\R$ defined as $f(X)=1/c\mathbb{E}[(X_{-})^p] = 1/c\|X_{-}\|_p^p$,
which is a convex, continuous and monotone risk function, but not cash-invariant ($\Pi=1$).
Its conjugate function is for $X^* \in L^q$ given by  (see \cite{Bot-Lorenz})
$$f^{*}(X^{*})=\left\{
\begin{array}{ll}
\frac{p-1}{c}\left \|\frac{c}{p}X^{*}\right \|_{q}^{q}, & \mbox {if } X^{*}\in -L^{q}_{+},\\
+\infty, & \mbox{otherwise}.
\end{array}\right.$$
Since $\dom f = L^p$, the \emph{cash-invariant hull} of $f$ has for all $X \in L^p$ the following formulation (see also Remark \ref{remppi})
$$f_{L^{p}_{+},1}(X)= f_1(X) = \max_{\substack{X^{*}\in
-L^{q}_{+}\\\mathbb{E}(X^{*})=-1}}\left\{\mathbb{E}(X^*X) -
\frac{p-1}{c}\left\|\frac{c}{p}X^{*}\right\|_{q}^{q}\right\}.$$
\end{example}

\begin{example}\label{ex45}
The next risk function we consider is the \emph{exponential risk measure} defined for $p \in [1,\infty]$ as being
$f : L^p \rightarrow \R$,  $f(X)=\mathbb{E}(\exp(-X))-1$. This is a convex, continuous and monotone, but not cash-invariant $(\Pi=1)$ risk function.
The conjugate function of $f$ is for $X^* \in (L^p)^*$ given by
$$f^{*}(X^{*})=\sup_{X\in L^{p}}\{\langle X^*,X \rangle-\mathbb{E}(\exp(-X))+1\},$$
which by the interchangeability of minimization and
integration (see \cite[Theorem 14.60]{Rock-Wets}) becomes (we make use again of the convention $0\ln(0) =0$)
$$\mathbb{E}\left \{\sup_{x\in\R}\{X^*x-\exp(-x)+1\} \right\} =\left\{
\begin{array}{ll}
\mathbb{E}[-X^{*}\ln(-X^*)+X^*]+1, & \mbox{if} \  X^* \leq 0,\\
+\infty, & \mbox{otherwise}.
\end{array}\right.$$
Consequently, one obtains via Remark \ref{remppi} for all $X \in L^p$ the following representation for the \emph{cash-invariant hull} of $f$
$$f_{L^{p}_{+},1}(X)= f_1(X) = \max_{\substack{X^{*}\in
(L_+^p)^{*}\\\mathbb{E}(X^{*})= 1}}\mathbb{E}[-X^*X - X^{*}\ln(X^*)].$$
\end{example}

\begin{example}\label{ex46}
For $p=\infty$ the so-called \emph{logarithmic risk
measure} $f:L^\infty \rightarrow \B$,
$$f(X)=\left\{
\begin{array}{ll}
\mathbb{E}(-\ln(X))-1, & \mbox {if } X>0,\\
+\infty, & \mbox {otherwise},
\end{array}\right.$$
is a convex, lower semicontinuous and monotone risk function which fails to be cash-invariant ($\Pi=1$).
Its conjugate function is given for $X^* \in (L^\infty)^*$ by
\begin{eqnarray*}
f^{*}(X^{*})=\sup_{\substack{X > 0}}\{\langle X^*,
X \rangle+\mathbb{E}(\ln(X)+1)\}\
\end{eqnarray*}
and can be further calculated by using \cite[Theorem 14.60]{Rock-Wets}. Indeed, one has
$$
f^{*}(X^{*})=\mathbb{E}\left\{\sup_{\substack{x > 0}}\{X^{*}x+\ln(x)+1\}\right\}=\left\{
\begin{array}{ll}
-\mathbb{E}(\ln(-{X^*})), & \mbox {if } X^*<0,\\
+\infty, & \mbox{otherwise}.
\end{array}\right.
$$
Before giving a dual representation for the cash-invariant hull of the logarithmic risk measure, one should notice that we are now in a situation where
\eqref{qcsecsp41} fails, but \eqref{qcsecsp42} is valid. Consequently, the \emph{cash-invariant hull} of $f$ can be for all
$X \in L^\infty$ given by
\begin{equation*}
f_{L^{\infty}_{+},1}(X)= f_1(X) = \max_{\substack{X^* \in (L^\infty)^*, X^{*} > 0
\\ \mathbb{E}(X^*)= 1}}\mathbb{E}[-X^*X + \ln(X^*)].
\end{equation*}
\end{example}

\section{The situation of missing lower semicontinuity}\label{sec4}

In the following we deal with the same problem of furnishing dual representations
for the monotone and cash-invariant hull of a convex risk function by using the
duality approach developed in Section \ref{sec3}, treating the particular case of a risk function which
\emph{fails to be lower semicontinuous}. We also discuss the difficulties which can arise when this topological assumption is missing.

For $p \in [1,\infty]$ consider $f:L^{p}\to\B$ defined by
\begin{equation*}f(X)=\left\{
\begin{array}{ll}
\|X-\mathbb{E}(X)\|_{p}, & \mbox {if } X_{-}\in L^{\infty},\\
+\infty, & \mbox{otherwise}.
\end{array}\right. \label{example}
\end{equation*}
This risk function is convex and fails to be lower semicontinuous for $p\in[1,\infty)$. One can easily verify that
$\dom f= L^{\infty}+L^{p}_{+}$.

Like in the previous section we take as ordering cone $L^p_+$, but work with a \emph{not necessarily constant numeraire} $\Pi \in L^p
\setminus \{0\}$. Our goal is to furnish a dual representation for the \emph{monotone $\Pi$-invariant hull} of $f$. To this end we will make use of the conjugate formula of $Y \mapsto \|Y-\mathbb{E}(Y)\|_{p}$, $p \in [1,\infty]$, which looks for $X^* \in (L^p)^*$ like (see \cite[Fact 4.3]{Bot-Lorenz})
\begin{equation}\label{eq40}
(\|\cdot-\mathbb{E}(\cdot)\|_{p})^{*}(X^*)=\!\left\{\!\!
\begin{array}{ll}
0, & \mbox {if} \ \!\exists Y^{*}\in (L^p)^*,\!\|Y^{*}\|_{(L^p)^*} \leq 1, \ \!\mbox{such that}\!\ X^* = Y^{*}-\mathbb{E}(Y^{*}),\\
+\infty, & \mbox{otherwise}.
\end{array}\right.
\end{equation}

\emph{The case $p = \infty$}. In this situation $\dom f = L^\infty$, $f$ is a convex and \emph{continuous} function and one can, consequently, use the qualification condition \eqref{qcsecsp41}, which is obviously fulfilled. Thus for the monotone $\Pi$-invariant hull of $f$ one can employ again formula \eqref{dualrep}.
This means that, by taking into consideration \eqref{eq40}, the \emph{monotone $\Pi$-invariant hull} of $f$ looks for all $X \in L^\infty$ like
\begin{equation}\label{eq41}
f_{L^{\infty}_{+},\Pi}(X) =\max_{\substack{\|Y^{*}\|_{(L^\infty)^*} \leq 1, \mathbb{E}(Y^*) - Y^* \in (L^\infty_+)^*\\
\mathbb{E}(Y^*\Pi) - \mathbb{E}(Y^*)\mathbb{E}(\Pi) + 1 = 0}}
\mathbb{E}(Y^*X)-\mathbb{E}(Y^*)\mathbb{E}(X).
\end{equation}
One can easily notice that if $\Pi$ is a \emph{constant numeraire}, then $f_{L^{\infty}_{+},\Pi} \equiv -\infty$.

\emph{The case $p \in [1,\infty)$}. In this second case we will proceed as follows: we first establish the monotone hull of $f$, along with a dual representation for it, then we discuss which are the difficulties that appear when trying to determine the dual representation of $f_{L^{p}_{+},\Pi}$. Recall that $f_{L^p_+,\Pi}(X) = (f_{L^p_+})_\Pi(X)$ for all $X \in L^p$. In this setting we denote the dual space of $L^p$ with $L^q$, $q=p/(p-1)$ (with the convention $1/0 = \infty$) and the same applies for the corresponding norm.

As $\dom f_{L^p_+} = \dom f + L^p_+ = L^\infty + L^p_+$, for every $X$ outside this set one has $f_{L^p_+}(X) = +\infty$. For $X \in L^\infty + L^p_+$ we have
\begin{equation}
f_{L^{p}_{+}}(X)=\inf_{\substack{Y \in L^\infty + L^p_+\\ Y-X \in -L^p_+}}\|Y-\mathbb{E}(Y)\|_{p} \label{primala ex}
\end{equation}
and, obviously, $f_{L^{p}_{+}}(X) \geq 0$. On the other hand, since $X = Z + Y$ for $Z \in L^\infty$ and $Y \in L^P_+$, it holds $X \geq \esinf Z$, thus $\esinf Z$ is feasible for the optimization problem in the right-hand side of \eqref{primala ex}, which means that $f_{L^{p}_{+}}(X) = 0$. Consequently,
$f_{L^p_+} = \delta_{L^\infty + L^p_+}$.

Before furnishing the monotone $\Pi$-invariant hull of $f$, let us shortly investigate how one could give dual representation for $f_{L^{p}_{+}}$. For $X \in L^\infty + L^p_+$ fixed one has to consider the convex optimization problem
\begin{equation}
\inf_{\substack{Y \in L^\infty + L^p_+\\ Y-X \in -L^p_+}}\|Y-\mathbb{E}(Y)\|_{p} \label{primsec4}
\end{equation}
and its \emph{Lagrange dual problem} (notice that $L^\infty$ is dense in $L^p$)
$$\sup_{\substack{X^{*}\in L^{q}_{+}}}\inf_{\substack{Y\in
L^\infty + L^p_+}}\{\|Y-\mathbb{E}(Y)\|_{p}+\langle X^{*},Y-X\rangle\} = \sup_{\substack{X^{*}\in L^{q}_{+}}}\inf_{\substack{Y\in L^p}}\{\|Y-\mathbb{E}(Y)\|_{p}+\langle X^{*},Y-X\rangle\}$$
or, equivalently,
\begin{equation}
\sup_{\substack{X^{*}\in -L^{q}_{+}}}\{\langle X^*,X \rangle - (\|\cdot-\mathbb{E}(\cdot)\|_{p})^{*}(X^*)\} = \sup_{\substack{\|Y^{*}\|_{q}\leq1,
\mathbb{E}(Y^*)-Y^*\in L^{q}_{+}}} \mathbb{E}(Y^{*}X)-\mathbb{E}(Y^*)\mathbb{E}(X).\label{dualsec4}
\end{equation}
In order to show that for the primal-dual pair \eqref{primsec4}-\eqref{dualsec4} strong duality holds, one needs to use the quasi-relative interior-type condition $(QC_4)$. Indeed, \eqref{primsec4} is of the same type as the problem $(P)$ from the preliminary section, when taking ${\cal X}={\cal Z}=S=L^p$, $K=L^p_+$, $g:L^p \rightarrow L^p$, $g(Y) = Y-X$ and having as objective function $f:L^{p}\to\B$,
\begin{equation*}f(Y)=\left\{
\begin{array}{ll}
\|Y-\mathbb{E}(Y)\|_{p}, & \mbox {if } Y_{-}\in L^{\infty},\\
+\infty, & \mbox{otherwise}.
\end{array}\right.
\end{equation*}
Since $X \in L^\infty + L^p_+$, for $X':=X-1 \in \dom f$ one has $g(X') \in -\qri L^p_+$ (see \cite{Borwein-Lewis}), while obviously $\cl(L^p_+ - L^p_+) = L^p$. More than that, as $(g(\dom f \cap S) + L^p_+) - (g(\dom f \cap S) + L^p_+) = L^p$ and the optimal objective value of \eqref{primsec4} is $f_{L^{p}_{+}}(X) = 0$, in order to show that $(QC_4)$ is verified, it is enough to prove that $(0,0)\notin \qi(\mathcal{E}_{f_{L^{p}_{+}}(X)})$, where
$$\mathcal{E}_{f_{L^{p}_{+}}(X)} = \{(\|Y-\mathbb{E}(Y)\|_{p} + \epsilon, Y-X+Z): Y\in L^{\infty}+L^{p}_{+}, Z\in L^{p}_{+}, \epsilon \geq 0\}.$$
Indeed, $(-1,0) \in N_{\mathcal{E}_{f_{L^{p}_{+}}(X)}}(0,0)$ and via Proposition \ref{caracterizari int}(ii) we get the desired conclusion. The qualification condition being verified it follows that
$$f_{L^{p}_{+}}(X) = \max_{\substack{\|Y^{*}\|_{q}\leq1,
\mathbb{E}(Y^*)-Y^*\in L^{q}_{+}}} \mathbb{E}(Y^{*}X)-\mathbb{E}(Y^*)\mathbb{E}(X)$$
and so one obtains for the monotone hull  of $f$ for all $X \in L^p$ the following dual representation
$$f_{L^{p}_{+}}(X)=\left\{
\begin{array}{ll}
 \max\limits_{\substack{\|Y^{*}\|_{q}\leq1,
\mathbb{E}(Y^*)-Y^*\in L^{q}_{+}}} \mathbb{E}(Y^{*}X)-\mathbb{E}(Y^*)\mathbb{E}(X), & \mbox {if } X\in L^{\infty}+L^{p}_{+},\\
+\infty, & \mbox{otherwise}.
\end{array}\right.$$

The \emph{monotone $\Pi$-invariant hull} of $f$ is the $\Pi$-invariant hull of $f_{L^p_+}$ and for its derivation we use the direct formulation of the latter, $f_{L^p_+}= \delta_{L^\infty + L^p_+}$, as it is easier to handle with. For all $X \in L^P$ the monotone $\Pi$-invariant hull of $f$ is
$$f_{L^{p}_{+},\Pi}(X)=\inf_{a\in\R}\{f_{L^{p}_{+}}(X-a\Pi)-a\} = \inf\limits_{\substack{(Y,a) \in (L^\infty + L^p_+) \times \R \\Y+a\Pi-X=0}} -a.$$
Since $f_{L^{p}_{+},\Pi}$ is the optimal objective value of a convex optimization problem, it is natural to ask if a dual formulation for it, via the duality theory, can be provided. Unfortunately, \emph{we are not always able to answer this question}. What we can say is, that for $X \in L^\infty + L^p_+ + \R\Pi = \dom f_{L^{p}_{+},\Pi}$ it holds
$f_{L^{p}_{+},\Pi}(X) = +\infty$. For $X \notin L^\infty + L^p_+ + \R\Pi$ one get as \emph{Lagrange dual problem} to
\begin{equation}\label{primalinvariant}
\inf\limits_{\substack{(Y,a) \in (L^\infty + L^p_+) \times \R\\Y+a\Pi-X=0}} -a
\end{equation}
the following optimization problem
$$\sup_{X^* \in L^q}\inf\limits_{(Y,a) \in (L^\infty + L^p_+) \times \R} [-a + \langle X^*, Y+a\Pi-X\rangle],$$
which, since $L^\infty$ is dense in $L^p$, is nothing else than
\begin{eqnarray}\label{dualinvariant}
\sup_{X^* \in L^q} \left [-\langle X^*,X \rangle + \inf_{a \in \R} a(\langle X^*,\Pi \rangle -1) + \inf_{Y \in L^p} \langle X^*, Y\rangle \right] =  -\infty
\end{eqnarray}
Nevertheless, we cannot be sure that this is the value which $f_{L^{p}_{+},\Pi}(X)$ takes, since no known qualification condition can be verified for \eqref{primalinvariant}-\eqref{dualinvariant}. This applies as well as for the classical generalized interior ones ($L^\infty + L^p_+$ is not closed) as for the one of quasi-relative interior-type. This emphasizes the fact that one can have exceptional situations for which the approach we use is, unfortunately, not suitable.

Let us also mention that whenever $\Pi \in L^\infty$ (which includes the situation when $\Pi$ is a \emph{constant numeraire}), then for all $a \in \R$ there exists $Y \in L^\infty + L^p_+$ such that $X = a\Pi + Y$ and so $f_{L^{p}_{+},\Pi}(X) = -\infty$. In this case we have for all $X \in L^p$
$$f_{L^{p}_{+},\Pi}(X)=\left\{
\begin{array}{ll}
-\infty, & \mbox {if } X\in L^{\infty}+L^{p}_{+} + \R\Pi,\\
+\infty, & \mbox{otherwise}.
\end{array}\right.$$

\begin{remark}\label{ctype}
The fact that $L^\infty + L^p_+$ is not closed does not make the applicability of the other main class of qualification conditions, the \emph{closedness-type} ones, for the convex optimization problem in \eqref{primalinvariant} possible, too.
\end{remark}

\end{document}